\documentclass[12pt,fleqn]{amsartvera}
\usepackage{amsmath}
\usepackage{amssymb}

\usepackage[margin=2cm]{geometry}
\usepackage[utf8]{inputenc}
\usepackage{amsfonts}
\usepackage{multirow}
\usepackage{lscape}
\usepackage{float}
\usepackage{graphicx}
\usepackage{hyperref}
\usepackage[usenames]{color}
\hypersetup{
    colorlinks,
    citecolor=blue,
    filecolor=blue,
    linkcolor=blue,
    urlcolor=blue
}

\newtheorem{theorem}{Theorem}[section]
\newtheorem{lemma}[theorem]{Lemma}
\newtheorem{propo}[theorem]{Proposition}

\newtheorem{definition}[theorem]{Definition}

\newtheorem{remark}[theorem]{Remark}

\title{On Bott-Morse Foliations and their Poisson structures in dimension $3$}

\author{M. Evangelista-Alvarado}
\address{Instituto de Matem\'aticas - Universidad Nacional Aut\'onoma de M\'exico\\Circuito Exterior, Ciudad Universitaria\\Coyoac\'an, 04510\\Mexico City\\Mexico}

\email{miguel.eva.alv@matem.unam.mx}

\author{P. Su\'arez-Serrato}

\address{Department of Mathematics, University of California Santa Barbara, CA, USA. \\ On leave from; Instituto de Matem\'aticas - Universidad Nacional Aut\'onoma de M\'exico\\Circuito Exterior, Ciudad Universitaria\\Coyoac\'an, 04510\\Mexico City\\Mexico}

\email{pablo@im.unam.mx}
	
\author{J. Torres Orozco}

\address{Centro de Ciencias Matem\'aticas\\ Antigua Carretera a P\'atzcuaro 8701\\
Col. Ex Hacienda San José de la Huerta, 58089\\ Morelia, Michoacán\\ Mexico}
\email{jonatan@cimat.mx}

\author{R. Vera}
\address{Instituto de Matem\'aticas - Universidad Nacional Aut\'onoma de M\'exico\\Circuito Exterior, Ciudad Universitaria\\Coyoac\'an, 04510\\Mexico City\\Mexico}

\email{vera@im.unam.mx, rvera.math@gmail.com}

\date{\today}

\begin{document}

\begin{abstract}
We show that a Bott-Morse foliation in dimension $3$ admits a linear, singular, Poisson structure of rank $2$ with Bott-Morse singularities. We provide the Poisson bivectors for each type of singular component, and compute the symplectic forms of the characteristic distribution. 
\end{abstract}
\maketitle

\section{Introduction}

The study of foliations on $3$-manifolds has had considerable influence on the direction of low dimensional topology. Early on Lickorish \cite{L65}, and independently Novikov \cite{N65}, and Zieschang, showed that every $3$-manifold admits a codimension-one foliation. The relevance of foliations has been comprehensively presented by Calegari \cite{C07}. The relationship between foliation theory and other topics of $3$-manifolds is still being explored (e.g. \cite{B16,CDMV16,E16,V16}). Various analytic, geometric, and topological complications arise when singularities are allowed to exist in a foliation. In this note we investigate foliations with exclusively Bott-Morse singularities in the context of Poisson geometry, continuing the foundational work in this direction by Sc\'ardua and Seade \cite{SS09, SS11}. These foliations have singularities that are modeled locally by Bott-Morse functions.\\ 

{\bf Example 1:} On the unit sphere $S^3$ inside ${\bf R}^4$ let $f$ be a Morse function defined using a height function with respect to an axis. The level sets of $f$ form a foliation of $S^3$, with leaves that are $2$-spheres and two singular polar points. This is one of the prototypes of a Bott-Morse foliation. \\

It has been suggested that it would be interesting to comprehend the Poisson manifolds equipped with these kinds of singularities (see Example 4 \cite{SS09}). We contribute to this circle of ideas with:

\begin{theorem}
\label{thm:Bott-Morse}
Let $M$ be a closed, orientable, connected smooth manifold of dimension $3$ equipped with a  codimension-one foliation $\mathcal{F}$ with Bott-Morse singularities. Then there exists a Poisson structure on $M$ of maximal rank $2$ supported on $\mathcal{F}$ which vanishes precisely at the Bott-Morse singularities. The associated bivectors are linear and can be found in table (\ref{tablefour}). The induced symplectic forms on the leaves are given by:
$$\frac{x_{1}}{k\sqrt{x_{1}^{2}+x_{2}^{2}}}\omega_{area}(q)$$
Here $k$ is a non-vanishing function (see \S \ref{S6}), $\omega_{area}$ denotes the canonical area form of the Euclidean plane, and $(x_1, x_2)$ are coordinates on the leaves of $\mathcal{F}$. If $\mathcal{F}$ is compact, then the Poisson structure obtained is complete.
\end{theorem}

Notice that this result provides a conformal family of Poisson structures, as the function $k$ varies. When such a foliation on $M$ with Bott-Morse singularities is transversally orientable and has no saddle-connections (see definition \ref{def:BM-foliation}), then its Poisson bivector vanishes only on two points or two circles, corresponding exactly to the singularities of the foliation. In the proof we combine methods to determine the bivectors and symplectic forms of a Poisson structure for which the associated Casimir functions can be described explicitly together with a gluing construction that collects all the possible local Poisson structures into a global one. This involves using the work of the second and fourth named authors \cite{GSV}, as we rely on certain $S^1$-invariant Poisson structures of dimension 4 in our arguments. Restricting these to 3-dimensional slices we find the Poisson structures associated to the original Bott-Morse foliation.\\

{\bf Example 2:} The abundance of foliations with Bott-Morse singularities in dimension $3$ can be seen with the help of Heegard splittings. A $3$--manifold that admits a genus--$g$ Heegard splitting for $g>0$ supports closed Bott–Morse foliations with $2g$ non-isolated center components, $2g-2$ isolated saddles, and leaves of genus in the range $1,\ldots, g$ (see \S 2.1 of \cite{SS11}). Furthermore, recall that a $3$--manifold with a genus--$g$  Heegard splitting also admits Heegard splittings for all genera $g'$ greater than $g$. Therefore it supports Bott–Morse foliations with genus--$g'$ leaves for all $g'>g$ as well.\\

Let us mention a few relevant relationships to put our result into the context of Poisson geometry. It was shown by Ibort and Mart{\'i}nez-Torres that every $3$--manifold admits a regular Poisson structure of rank 2 \cite{IM03}. The two examples described above illustrate how every $3$--manifold admits a foliation with Bott-Morse singularities. As a consequence, our main result allows us to find an associated singular Poisson structure of generic rank 2. Moreover, our result provides a quantitative perspective as we provide explicit formul{\ae} for the local forms. Poisson structures related to local fibrations have also been of recent interest, see for example Avenda{\~n}o-Vorobiev \cite{AV17}. As the structures we find are linear, there are Lie algebras associated to some of them, which can be compared to, for example, the descriptions of Lie-Poisson structures of Ginzburg-Weinstein \cite{GW92}. In particular, the celebrated linearization result of Conn \cite{C84} is superseded by our linear normal forms (see also Crainic-Fernandes \cite{CF11}). The equivalence classes of the Poisson structures described here can be understood in terms of weak Morita equivalence. As the genus of the Heegard splitting in example 2 increases, the fundamental groups of the leaves of the associated Poisson structure change, and therefore they are Morita inequivalent in the sense of Bursztyn-Weinstein \cite{BW05}.\\

After reviewing the definitions and background needed for our arguments in \S\ref{S2} and \S\ref{S3}, we proceed to describe the Poisson bivectors in \S\ref{sec:local} and symplectic forms associated to Bott-Morse foliations in \S\ref{S5}. These data come together in \S\ref{S6} to complete the global Poisson structure and the proof our main result. We include in  \S\ref{S7} some restrictions to the existence of compatible Poisson structures on Bott-Morse foliations in higher dimensions, pointing to potential extensions of this line of research. We end with a remark in \S\ref{S8} about the Poisson cohomology of these structures in the case of homogenous linear and quadratic coefficients are presented in tables \ref{tableseven} and \ref{tableeight}.\\ 

{\bf Acknowledgements:} PSS acknowledges \& thanks support from UNAM-DGAPA-PAPIIT-IN102716 and UC-MEXUS CN-16-43. RV thanks UNAM-DGAPA. JTO thanks support from FORDECyT 265667.

\section{Bott-Morse foliations}\label{S2}

This section follows notations used in \cite{SS09} and \cite{SS11}, where Bott-Morse foliations on dimension $3$ were described. 
\newline

Let $M^m$ be a closed, orientable, smooth manifold of dimension $m$, for $m\geq 3$. Let $\mathcal{F}$ be a codimension-one smooth foliation with singularities on $M$. Denote by ${\rm Sing}(\mathcal{F})$ the set of singular points of $\mathcal{F}$.

\begin{definition}
A smooth function $f: M\to {\bf R}$ is said to be a {\em Bott-Morse function} if the following conditions hold:
\begin{itemize}
\item [i)] The critical set of $f$ is a disjoint union of closed, connected embedded submanifolds $N_j$:
\[{\rm Crit}(f)=\sqcup_{j=1}^{t}N_{j}.\]
Such submanifolds are referred to as the {\em critical components} of the Bott-Morse function. 

\item [ii)] The function $f$ is non-degenerate when restricted to any critical component. 

\end{itemize}

\end{definition}

The non-degeneracy condition for a Bott-Morse function means that for each $p\in N_{j}$ and a small disc $\Sigma_{p}$ transversal to $N_{j}$, of complementary dimension, the restriction $f|_{\Sigma_{p}}$ is a Morse function with $p$ a Morse singularity. 

\begin{definition}
The singularities of the foliation $\mathcal{F}$ are called {\em Bott-Morse singularities} if:
\begin{itemize}
\item [i)] The singularity set can be decomposed as:
$${\rm Sing}(\mathcal{F})=\sqcup_{j=1}^{t}N_j$$
Here $N_j$ is a closed, connected submanifold of $M$ with ${\rm codim}(N_j)\leq 2$.

\item [ii)] In a neighborhood of each singular point, $\mathcal{F}$ is defined by a Bott-Morse function.
\end{itemize}
\end{definition}

Let $p\in N_j$ be a Bott-Morse singularity and $n_j$ be the dimension of the critical component $N_j$. Then there exist a neighborhood $V_p\subset M$ and a foliation $\mathcal{G}$, such that the restriction of $\mathcal{F}$ to $V_p$ is a product foliation $P\times \mathcal{G}$, for some disc $P\subset{\bf R}^{n_{j}}$. The foliation $\mathcal{G}$ is defined on a disc $D\subset{\bf R}^{m-n_{j}}$ whose fibers are given by a Morse function. This implies the existence of a local diffeomorphism $\varphi:V_p\to P\times D$. Said otherwise, we have that:
\begin{itemize}
\item ${\rm Sing}(\mathcal{F})\cap V_p = N_{j}\cap V_p$.
\item $\varphi(N_{j}\cap V_p) = P\times \{0\}\subset{\bf R}^{n_{j}}\times{\bf R}^{m-n_{j}}$. 
\item There exist local coordinates $(\bar{x},x)=(\bar{x}_{1},...,\bar{x}_{n_{j}}, x_{1},...,x_{m-n_{j}})\in V_p$ such that $N_{j}\cap V_p$ is defined by $\{(x_{1},...,x_{m-n_{j}})=0\}$
and $\mathcal{F}|_{V_p}$ is given by the level sets of a Morse function $J_{N_{j}}(\bar{x},x)=\Sigma_{j=1}^{m-n_{j}}\lambda_{j}x_{j}^{2}$, where $\lambda_j=\pm 1$.
\end{itemize} 

The discs $\Sigma_{p}=\varphi^{-1}(x(p\times D))$ are transverse to $N_{j}$, outside ${\rm Sing}(\mathcal{F})$. Denote by $\mathcal{G}(N_{j})=\mathcal{F}|_{\Sigma_{p}}$, the {\em transverse type} of $\mathcal{F}$ along $N_{j}$. It is a codimension-one foliation on $\Sigma_{p}$ with an ordinary Morse singularity at $\{p\}=N_{j}\cap \Sigma_{p}$. The Morse index is constant in $N_j$.

\begin{definition}
A critical component $N_{j}\subset {\rm Sing}(\mathcal{F})$ is called:
    \begin{enumerate}
        \item A {\bf Center} if the transverse type $\mathcal{F}|_{\Sigma_{p}}$ of $\mathcal{F}$ along $N_{j}$ is a center, that is, the Morse singularity $p$ has Morse index $0$ or $m-n_{j}$.
        \item A {\bf Saddle} if the transverse type $\mathcal{F}|_{\Sigma_{p}}$ is a saddle, that is, if the Morse singularity has Morse index different from $0$ or $m-n_{j}$.
    \end{enumerate}
\end{definition}

Given a saddle component $N_{j}\subset$ ${\rm Sing}(\mathcal{F})$, a {\em separatrix} of $N_j$ is a leaf $L$ of the foliation $\mathcal{F}$ such that its closure $\bar{L}$ contains $N_{j}$. This means that $L$ meets each small disc $\Sigma$ in ${\bf R}^{(m-n_{j})}$, which is transversal to $N_{j}$  in a separatrix of $\mathcal{F}|_{\Sigma}$.  In a neighborhood of $N_j$, the separatrices through $p$ are given by the relation:  $x^{2}_{1}+...+x^{2}_{r}=x^{2}_{r+1}+...+x^{2}_{n_{j}}\neq 0$, where $r$ is the Morse index. In a neighborhood of a center component the leaves of $\mathcal{F}$ are diffeomorphic to spheres $S^{(m-(n_{j}+1))}$.
\newline

 We say that $\mathcal{F}$ has a {\it saddle-connection} if there exist saddle components $N_{1}$, $N_{2}$, $N_{1}\neq N_{2}$, and a leaf $L$ of $\mathcal{F}$ which is simultaneously a separatrix of $N_{1}$ and $N_{2}$. If a leaf $L$ is a separatrix of $\mathcal{F}$ through $N$ and $L$ meets some transversal disc $\Sigma$ in two distinct separatrices of $\mathcal{F}|_{\Sigma}$ then we say $L$ is a {\it self-saddle-connection} of $\mathcal{F}$.    

\begin{definition}\cite{SS11}\label{def:BM-foliation}
We say that $\mathcal{F}$ is a Bott-Morse foliation if:
    \begin{enumerate}
        \item The singularities of $\mathcal{F}$ are of Bott-Morse type,
        \item $\mathcal{F}$ is transversally orientable, and
        \item $\mathcal{F}$ has no saddle-connections on $M$.
    \end{enumerate}
\end{definition}

If $M$ has dimension $3$, then the dimension of $N_j$ can be $0$ or $1$. If $\dim(N_j)=0$ there are two possible center singularities and two saddle singularities. If $\dim(N_j)=1$ there are two possible center singularities and one possible saddle singularity.

\begin{remark}
A foliation $\mathcal{F}$ is said to be {\em compact} if every leaf is compact. In this case there are no saddle components  [Proposition 1, \cite{SS09}]. If every leaf of $\mathcal{F}$ is closed off $\textnormal{Sing}(\mathcal{F})$, then $\mathcal{F}$ is said to be {\em closed}.  
\end{remark}

\section{Poisson structures}\label{S3}

We will now include the facts needed to understand the construction of Poisson structures with Bott-Morse singularities. The Schouten-Nijenhuis bracket is an operation which extends the Lie derivative on multi-vector fields $\left[ \cdot ,  \cdot \right]_{\mathrm{SN}} \colon \mathfrak{X}^{p}(M)  \times \mathfrak{X}^{q}(M) \rightarrow \mathfrak{X}^{p+q-1}(M)$. Among its numerous applications, it plays a fundamental role in Poisson Geomety.  

\begin{definition}
A {\em Poisson bivector}, or a {\em Poisson structure} on $M$ is a bivector field  $\pi \in \Gamma (\Lambda^2 TM )= \mathfrak{X}^2(M)$ satisfying $\left[ \pi, \pi \right]_{\mathrm{SN}}=0$. 
\end{definition}

Note  that every manifold $M$ admits a trivial Poisson structure by defining $\pi=0$ at every point $p\in M$. A class of non-trivial Poisson structures are symplectic manifolds. A Poisson structure can also be defined in terms of a bracket $\{ ,\}$ on $C^{\infty}(M)$. It satisfies a derivation rule, and it endows $C^{\infty}(M)$ with a Lie algebra structure. It follows that the bracket $\{g,h\}$ depends solely on the first derivatives of the functions $g$ and $h$. The Poisson bivector and the bracket are related by
\begin{equation*}\label{eq:bracket-bivector}
\{g,h\}=\pi(dg,dh).
\end{equation*}

The Poisson bivector satisfies the properties of bilinearity, skew-symmetry, and the Leibniz identity, which are defined by the bracket. The vanishing of the Schouten-Nijenhuis bracket of the bivector $\pi$ with itself corresponds to the Jacobi identity of the Poisson bracket. 
Nevertheless, in this paper we will only use the description of a Poisson structure through a bivector field. 

A Poisson bivector $\pi$ can be described locally, for coordinates $(x^1, \dots , x^n)$;
\begin{equation*}
\pi(x)=\frac{1}{2}\sum_{i,j=1}^n\pi^{ij}(x)\frac{\partial}{\partial x^i}\wedge \frac{\partial}{\partial x^j}.
\end{equation*}
Here $\pi^{ij}(x)=\pi(dx^i,dx^j)=-\pi(dx^j, dx^i)$.
\newline

Given a bivector  $\pi$ on $M$, a point $q\in M$, and  $\alpha_q\in T_q^*M$  it is possible to define a bundle map $\mathcal{B}:T^*M\to TM$ given by:
\begin{equation}
\label{fun:Anchor}
\mathcal{B}_q(\alpha_q)(\cdot)=\pi_q(\cdot ,\alpha_q)
\end{equation}
We define the {\em rank} of  $\pi$ at  $q\in M$ to be equal to the rank of $\mathcal{B}_q:T^*_qM\to T_qM$. This is also the rank of the matrix $\pi^{ij}(x)$. 
If $\pi$ is a Poisson bivector and $h\in C^{\infty}(M)$ is a smooth function we define the {\em Hamiltonian vector field} $X_h$ by
$X_h=\mathcal{B}(dh)$. 
\medskip

For a point $x_o\in M$ define the linear subspace:
\begin{equation*}
\Gamma_{x_o}(M)=\{v\in T_{x_o}(M)\, | \, \exists  \, f \in C^{\infty}(M),  \, X_f(x_o)=v\}
\end{equation*}
Note that, $\Gamma_{x_o}(M)=\mathrm{Im}(\mathcal{B}_{x_o})$. The set $\Gamma(M)=\{\Gamma_{x_o}(M)\}$ is a differentiable distribution called the \textit{characteristic distribution} of the Poisson structure. If the rank of $\Gamma(M)$ is constant, we call it a {\it regular distribution}; else, it is called a {\it singular distribution}. 

\begin{theorem}[{\em Symplectic Stratification Theorem} \cite{DZ05}] The induced characteristic distribution $\Gamma(M)$ of the Poisson manifold $(M, \pi)$ is completely integrable, and the Poisson structure induces symplectic structures on the leaves $\Gamma_{x_o}$. This foliation is integrable in the sense of Stefan-Sussman.
\end{theorem}

The set $\Gamma_q$, the symplectic leaf of $M$ through the point $q$, is also the collection of points that may be joined via piecewise smooth integral curves of Hamiltonian vector fields. Write $\omega_{\Gamma_q}$ for the symplectic form on  $\Gamma_q$. Observe that $T_q\Gamma_q$ is exactly the characteristic distribution of $\pi$ through $p$. Therefore, given $u_q, v_q \in T_q\Gamma_q$
there exist $\alpha_q, \beta_q\in T^*_qM$ that under $\mathcal{B}_q$ go to $u_q$ and $v_q$. Using this we can describe $\omega_{\Gamma_q}$:
\begin{equation}
\label{E:Symp-form-gen}
\omega_{\Gamma_q}(q)(u_q, v_q)=\pi_q(\alpha_q, \beta_q)=\langle \alpha_q,  v_q \rangle=-\langle \beta_q,  u_q \rangle.
\end{equation}

As the rank varies, so do the dimensions of the symplectic leaves of the foliation.

\begin{definition}
A Poisson manifold $M$ is said to be complete if every Hamiltonian
vector field on $M$ is complete.
\end{definition}

Notice that $M$ is complete if and only if every symplectic leaf is bounded in the
sense that its closure is compact.

\begin{definition}
\label{D:Casimir}
Let $M$ be a Poisson manifold. A function $h\in C^\infty(M)$ is called a {\em Casimir} if $\mathcal{B}(dh)=0$.
\end{definition}
	
The following was shown in \cite{GSV}:
\begin{theorem}
\label{T:Const-Poisson}
Let $M$ be an orientable $n$-manifold, $N$ an orientable $n-2$ manifold, and $f:M\to N$ a smooth map.
Let $\mu$ and $\Omega$ be orientations of $M$ and $N$ respectively. The bivector $\pi$ on $M$ defined by
\begin{equation}
\label{E:Def-Intrinsic}
\pi(dg, dh)\mu=k\,dg\wedge dh \wedge f^*\Omega
\end{equation}
for any $g, h\in C^{\infty}(M)$,  where $k$ is any non-vanishing function on $M$ is Poisson. Moreover, its symplectic leaves are
\begin{enumerate}
\item[(i)] the 2-dimensional leaves $f^{-1}(s)$ where $s\in N$ is
a regular value of $f$, \item[(ii)] the 2-dimensional leaves
$f^{-1}(s)\setminus \{\mbox{Critical Points of $f$}\}$ where $s\in
N$ is a singular value of $f$. \item[(iii)] the 0-dimensional
leaves corresponding to each critical point.
\end{enumerate}
\end{theorem}

Equation (\ref{E:Def-Intrinsic}) is known as the Flaschka-Ratiu formula. It provides a way to construct Poisson manifolds with prescribed Casimirs.

\section{Local expressions for the Poisson structures}\label{sec:local}

In this section we give an explicit Poisson local structure in a neighborhood of the singularities of fibrations $F:M^{3}\times S^{1}\to {\bf R}\times S^{1}$, where $F|_{M^3}$ is a Bott-Morse function. It defines a foliation $\mathcal{F}$ with Bott-Morse singularities on $M$. We we will describe the general strategy to find Poisson local bivectors:
\begin{itemize}
    \item[Step 1:] Consider as Casimirs for the Poisson structure that we will find the functions $C_{1}$ and $C_{2}$ that describe locally the singularity of the fibration.
    \item[Step 2:] Calculate the differentials of the functions $C_{1}$ and $C_{2}$. 
    \item[Step 3:] Use the Flaschka-Ratiu formula (\ref{E:Def-Intrinsic}) to calculate the skew-symmetric matrix $\Pi$ with 
    matrix entries $\Pi_{ij} = dx_{i}\wedge dx_{j}\wedge d C_{1}\wedge d C_{2}$. Each $\Pi_{ij}$ is equal to the determinant det$(e_{i},e_{j},d C_{1},d C_{2})$, where $\{e_i\}_{i=1}^4$ is the canonical basis of ${\bf R}^4$, and they are considered as column vectors. \\
\newline 
The bivector $\Pi$ is the matrix of the endomorphism $\mathcal{B}$ associated with the Poisson structure, with $C_{1}$ and $C_{2}$ Casimirs. Then $\Pi$ annihilate the differentials $d C_{1}$ and $d C_{2}$.  

	\item[Step 4:] Write the Poisson bivector using the matrix found in the previous step.
\end{itemize}

\subsection{Local expressions near a Bott-Morse singularity}\label{Poisson}

Let $(p, t)\in M^3\times S^1$, for $p\in N_{j}\subset$ ${\rm Sing}(\mathcal{F})$, since $\Sigma_{p}$ is transversal to $N_{j}$. The following table (\ref{tabletwo}) contains the Casimirs in consideration, according to each component $N_{j}$, arranged according to the dimension of the component, their type, and their Morse index.

\begin{table}[H]
\centering
\begin{tabular}{|c|c|l|}
\hline
\multicolumn{3}{|l|}{$\dim(N_{j})=0$}                                                                     \\ \hline
Type                    & Morse Index        & \multicolumn{1}{c|}{Casimirs}                              \\ \hline
\multirow{4}{*}{Center} & \multirow{2}{*}{0} & $C_1(x_{1}, x_{2},x_{3},t) = x_{1}^{2}+x_{2}^{2}+x_{3}^{2}$  \\
                        &                    & $C_2(x_{1}, x_{2},x_{3},t) = t$                              \\ \cline{2-3} 
                        & \multirow{2}{*}{3} & $C_1(x_{1}, x_{2},x_{3},t) = -x_{1}^{2}-x_{2}^{2}-x_{3}^{2}$ \\
                        &                    & $C_2(x_{1}, x_{2},x_{3},t) = t$                              \\ \hline
\multirow{4}{*}{Saddle} & \multirow{2}{*}{1} & $C_1(x_{1}, x_{2},x_{3},t) = -x_{1}^{2}+x_{2}^{2}+x_{3}^{2}$ \\
                        &                    & $C_2(x_{1}, x_{2},x_{3},t) = t$                              \\ \cline{2-3} 
                        & \multirow{2}{*}{2} & $C_1(x_{1}, x_{2},x_{3},t) = -x_{1}^{2}-x_{2}^{2}+x_{3}^{2}$ \\
                        &                    & $C_2(x_{1}, x_{2},x_{3},t) = t$                              \\ \hline
\multicolumn{3}{|l|}{$\dim(N_{j})=1$}                                                                     \\ \hline
\multirow{4}{*}{Center} & \multirow{2}{*}{0} & $C_1(x_{1}, x_{2},x_{3},t) = x_{1}^{2}+x_{2}^{2}$            \\
                        &                    & $C_2(x_{1}, x_{2},x_{3},t) = t$                              \\ \cline{2-3} 
                        & \multirow{2}{*}{2} & $C_1(x_{1}, x_{2},x_{3},t) = -x_{1}^{2}-x_{2}^{2}$           \\
                        &                    & $C_2(x_{1}, x_{2},x_{3},t) = t$                              \\ \hline
\multirow{2}{*}{Saddle} & \multirow{2}{*}{1} & $C_1(x_{1}, x_{2},x_{3},t) = -x_{1}^{2}+x_{2}^{2}$           \\
                        &                    & $C_2(x_{1}, x_{2},x_{3},t) = t$                              \\ \hline
\end{tabular}
\vspace{2mm}
\caption{Casimirs for each $N_{j}$.}
\label{tabletwo}
\end{table}

\noindent Then the corresponding differentials $dC_{1}$ and $dC_{2}$, for each case are found in the following table (\ref{tablethree}):

\begin{table}[H]
\centering
\begin{tabular}{|c|c|l|}
\hline
\multicolumn{3}{|l|}{$\dim(N_{j})=0$}                                                                   \\ \hline
Type                    & Morse Index        & \multicolumn{1}{c|}{Differentials}                       \\ \hline
\multirow{4}{*}{Center} & \multirow{2}{*}{0} & $dC_1(x_{1}, x_{2},x_{3},t) = (2x_{1},2x_{2},2x_{3},0)$    \\
                        &                    & $dC_2(x_{1}, x_{2},x_{3},t) = (0,0,0,1)$                   \\ \cline{2-3} 
                        & \multirow{2}{*}{3} & $dC_1(x_{1}, x_{2},x_{3},t) = (-2x_{1},-2x_{2},-2x_{3},0)$ \\
                        &                    & $dC_2(x_{1}, x_{2},x_{3},t) = (0,0,0,1)$                   \\ \hline
\multirow{4}{*}{Saddle} & \multirow{2}{*}{1} & $dC_1(x_{1}, x_{2},x_{3},t) = (-2x_{1},2x_{2},2x_{3},0)$   \\
                        &                    & $dC_2(x_{1}, x_{2},x_{3},t) = (0,0,0,1)$                   \\ \cline{2-3} 
                        & \multirow{2}{*}{2} & $dC_1(x_{1}, x_{2},x_{3},t) = (-2x_{1},-2x_{2},2x_{3},0)$  \\
                        &                    & $dC_2(x_{1}, x_{2},x_{3},t) = (0,0,0,1)$                   \\ \hline
\multicolumn{3}{|l|}{$\dim(N_{j})=1$}                                                                   \\ \hline
\multirow{4}{*}{Center} & \multirow{2}{*}{0} & $dC_1(x_{1}, x_{2},x_{3},t) = (2x_{1},2x_{2},0,0)$         \\
                        &                    & $dC_2(x_{1}, x_{2},x_{3},t) = (0,0,0,1)$                   \\ \cline{2-3} 
                        & \multirow{2}{*}{2} & $dC_1(x_{1}, x_{2},x_{3},t) = (-2x_{1},-2x_{2},0,0)$       \\
                        &                    & $dC_2(x_{1}, x_{2},x_{3},t) = (0,0,0,1)$                   \\ \hline
\multirow{2}{*}{Saddle} & \multirow{2}{*}{1} & $dC_1(x_{1}, x_{2},x_{3},t) = (-2x_{1},2x_{2},0,0)$        \\
                        &                    & $dC_2(x_{1}, x_{2},x_{3},t) = (0,0,0,1)$                   \\ \hline
\end{tabular}
\vspace{2mm}
\caption{Differentials of the Casimirs considered in table \ref{tabletwo}.}
\label{tablethree}
\end{table}

\noindent Each bivector has rank 2 and annihilates $dC_{1}$ and $dC_{2}$.  For simplicity we reduce our notation of bivector fields by $\partial_{ij}:=\partial_i\wedge\partial_j$ for $i<j$. The corresponding bivectors are given by the expressions in table (\ref{tablefour}):

\begin{table}[H]
\centering
\begin{tabular}{|c|c|r|}
\hline
\multicolumn{3}{|l|}{$\dim(N_{j})=0$}                                                                                                                    \\ \hline
Type                    & Morse Index        & \multicolumn{1}{c|}{Bivector}                                                                             \\ \hline
\multirow{4}{*}{Center} & \multirow{2}{*}{0} & \multirow{4}{*}{$\pi= k \left(x_{3} \partial_{12} - x_{2}\partial_{13} + x_{1}\partial_{23}\right) \quad (1)$} \\
                        &                    &                                                                                                           \\ \cline{2-2}
                        & \multirow{2}{*}{3} &                                                                                                           \\
                        &                    &                                                                                                           \\ \hline
\multirow{4}{*}{Saddle} & \multirow{2}{*}{1} & \multirow{2}{*}{$\pi= k \left(-x_{3} \partial_{12} +x_{2}\partial_{13} + x_{1}\partial_{23}\right) \quad (2)$} \\
                        &                    &                                                                                                           \\ \cline{2-3} 
                        & \multirow{2}{*}{2} & \multirow{2}{*}{$\pi= k \left(-x_{3} \partial_{12} +x_{2}\partial_{13} + x_{1}\partial_{23}\right) \quad (3)$} \\
                        &                    &                                                                                                           \\ \hline
\multicolumn{3}{|l|}{$\dim(N_{j})=1$}                                                                                                                    \\ \hline
\multirow{4}{*}{Center} & \multirow{2}{*}{0} & \multirow{4}{*}{$\pi= k \left(- x_{2}\partial_{13} + x_{1}\partial_{23}\right) \quad (4)$}                     \\
                        &                    &                                                                                                           \\ \cline{2-2}
                        & \multirow{2}{*}{2} &                                                                                                           \\
                        &                    &                                                                                                           \\ \hline
\multirow{2}{*}{Saddle} & \multirow{2}{*}{1} & \multirow{2}{*}{$\pi= k \left( x_{2}\partial_{13} + x_{1}\partial_{23}\right) \quad (5)$}                      \\
                        &                    &                                                                                                           \\ \hline
\end{tabular}
\vspace{2mm}
\caption{Bivector $\pi$ associated with the matrix $\Pi$.}
\label{tablefour}
\end{table}

\noindent Here $k=k(x_{1},x_{2},x_{3},t)$ is a nonzero smooth function on $M\times S^{1}$.\\

If $\pi$ presents one of the forms as in the Table \ref{tablefour}, then the above tensor can be interpreted as multiple of a linear Poisson structure in ${\bf R}^{3}$. Hence, up to the factor $k$, it is dual to the Lie algebra structure of real dimension three possessing commutation relations between the basis elements $e_{1}$, $e_{2}$ and $e_{3}$ that we will show below.

\begin{enumerate}
\item If $\pi$ is of the form $(1)$, then  
       \[
            \begin{array}{ccc}
                [e_{1},e_{2}]=e_{3}, & [e_{1},e_{3}]=-e_{2}, & [e_{2},e_{3}]=e_{1}. 
            \end{array}
        \]
         This Lie algebra is isomorphic to $\mathfrak{so}_{3}({\bf R}).$
\item If $\pi$ is of the form $(2)$, then 
      \[
            \begin{array}{ccc}
                [e_{1},e_{2}]=-e_{3}, & [e_{1},e_{3}]=e_{2}, & [e_{2},e_{3}]=e_{1}.
            \end{array}
        \]
        This Lie algebra is isomorphic to $\mathfrak{sl}_{2}({\bf R}).$
\item If $\pi$ is of the form $(3)$, then 
        \[
            \begin{array}{ccc}
                [e_{1},e_{2}]=0, & [e_{1},e_{3}]=-e_{2}, & [e_{2},e_{3}]=e_{1}. 
            \end{array}
        \]
        This Lie algebra is isomorphic to $\mathfrak{e}(2).$ 
\end{enumerate}
\begin{remark}
The Poisson structure constructed on $M^3\times S^1$ does not depend on $t\in S^1$. Then, the Poisson structure on $M^3$ is just the restriction of $\Pi$ to $M^3$. Notice that each and every one of the Poisson structures we found depend on a smooth non-vanishing function $k$. That is, we actually found a family of Poisson structures that changes with $k$. 
\end{remark}

\section{Symplectic forms on the leaves near singularities}\label{S5}

In the next section we describe the leaves of the characteristic distribution of the Poisson structures found in the previous Section \ref{Poisson}. We will present the symplectic forms for each component of the singularity set of a Bott-Morse foliation. First, let us explain the general procedure that we will follow. \\ 

\begin{itemize}
    \item[{\bf Step 1;}] Obtain the tangent vectors $u_{q}$ and $v_{q}$ to the symplectic leaf $\Gamma_{q}$ at $q\in M$ by computing the null space of the differentials $dC_{1}$ and $dC_{2}$.
    \item[{\bf Step 2;}] Use the local expressions of the Poisson bivectors, so one can find $\alpha_{q}$ such that $\mathcal{B}_{q}(\alpha_{q})=u_{q}$, similarly find  $\beta_{q}$ such that $\mathcal{B}_{q}(\beta_{q})=v_{q}$, for the bundle map (\ref{fun:Anchor}.)
    \item[{\bf Step 3;}] Calculate the symplectic form using equation (\ref{E:Symp-form-gen}):
    \[\omega_{\Gamma_{q}}(q)(u_{q},v_{q})=\langle \alpha_{q}, v_{q}\rangle=-\langle \beta_{q}, u_{q}\rangle\] 
\end{itemize}
\begin{propo}
Let $q=(x_{1},x_{2},x_{3},t)\in B^{3}\times S^{1}$ and $\pi$ one of the bivectors of the Table \ref{tablefour}. The symplectic form induced by $\pi$ on the symplectic leaf $\Gamma_{p}$ through at the point $q$ is given by: 
\begin{equation}
    \frac{x_{1}}{k(x_{1},x_{2},x_{3},t)\sqrt{x_{1}^{2}+x_{2}^{2}}}\omega_{area}(q)
\end{equation}
Here $\omega_{area}$ is the area form on $\Gamma_{p}$ induced by the euclidean metric on $B^{3}\times S^{1}$. 
\end{propo}

\begin{proof}
First assume $x_{1}^{2}+x_{2}^{2}\neq 0$ and recall that the Casimirs for $\pi$ depending on each case are shown in the table (\ref{tabletwo}).\\

\noindent The following table (\ref{tablefive}) contains the vectors $u_q, v_q$ tangent to each fiber $\Gamma_q$, for the different components:
\begin{table}[H]
\centering
\begin{tabular}{|c|c|c|c|}
\hline
\multicolumn{4}{|l|}{$\dim(N_{j})=0$}                                                                                                                                                                                                                                                                                  \\ \hline
Type                    & Morse Index        & $u_{q}$                                                                                        & $v_{p}$                                                                                                                                                                \\ \hline
\multirow{4}{*}{Center} & \multirow{2}{*}{0} & \multirow{4}{*}{$\frac{1}{\sqrt{x_{1}^{2}+x_{2}^{2}}}(-x_{2}\partial_{1}+x_{1}\partial_{2})$}  & \multirow{4}{*}{$\frac{-1}{x_{1}^{2}+x_{2}^{2}}(x_{1}^{2}x_{3}\partial_{1}+x_{1}x_{2}x_{3}\partial_{2})+x_{1}\partial_{3}$}                                            \\
                        &                    &                                                                                                &                                                                                                                                                                        \\ \cline{2-2}
                        & \multirow{2}{*}{3} &                                                                                                &                                                                                                                                                                        \\
                        &                    &                                                                                                &                                                                                                                                                                        \\ \hline
\multirow{4}{*}{Saddle} & \multirow{2}{*}{1} & \multirow{2}{*}{$\frac{1}{\sqrt{x_{1}^{2}+x_{2}^{2}}}(x_{2}\partial_{1}+x_{1}\partial_{2})$}   & \multirow{2}{*}{\begin{tabular}[c]{@{}c@{}}$\frac{-1}{x_{1}^{2}+x_{2}^{2}} (x_{1}^{2}x_{3}\partial_{1}-x_{1}x_{2}x_{3}\partial_{2})+x_{1}\partial_{3}$\end{tabular}} \\
                        &                    &                                                                                                &                                                                                                                                                                        \\ \cline{2-4} 
                        & \multirow{2}{*}{2} & \multirow{2}{*}{$\frac{1}{\sqrt{x_{1}^{2}+x_{2}^{2}}}(-x_{2}\partial_{1}+x_{1}\partial_{2})$}  & \multirow{2}{*}{$\frac{-1}{x_{1}^{2}+x_{2}^{2}}(x_{1}^{2}x_{3}\partial_{1}+x_{1}x_{2}x_{3}\partial_{2})+x_{1}\partial_{3}$}                                            \\
                        &                    &                                                                                                &                                                                                                                                                                        \\ \hline
\multicolumn{4}{|l|}{$\dim(N_{j})=1$}                                                                                                                                                                                                                                                                                  \\ \hline
\multirow{4}{*}{Center} & \multirow{2}{*}{0} & \multirow{4}{*}{$\frac{1}{\sqrt{x_{1}^{2}+x_{2}^{2}}}(-x_{2}\partial_{1}+x_{1}\partial_{2})$}  & \multirow{6}{*}{$x_{1}\partial_{3}$}                                                                                                                                   \\
                        &                    &                                                                                                &                                                                                                                                                                        \\ \cline{2-2}
                        & \multirow{2}{*}{2} &                                                                                                &                                                                                                                                                                        \\
                        &                    &                                                                                                &                                                                                                                                                                        \\ \cline{1-3}
\multirow{2}{*}{Saddle} & \multirow{2}{*}{1} & \multirow{2}{*}{$\frac{1}{\sqrt{x_{1}^{2}+x_{2}^{2}}}(x_{2}\partial_{1}+x_{1}\partial_{2})$} &                                                                                                                                                                        \\
                        &                    &                                                                                                &                                                                                                                                                                        \\ \hline
\end{tabular}
\vspace{2mm}
\caption{Tangent vectors to the fibers.}
\label{tablefive}
\end{table}

\noindent Notice that they are annihilated by $dC_{1}(q)$ and $dC_{2}(q)$. Moreover, we have chosen them to be orthogonal with respect to the euclidean 
metric $dx_{1}^{2}+dx_{2}^{2}+dx_{3}^{3}+dt^{2}$.\\

\noindent For each case, using the local expression of $\pi$, it is straightforward to check that $\mathcal{B}_{q}(\alpha_{q})=u_{q}$, for $\alpha_q$. \end{proof}  

\section{Global Poisson structure}\label{S6}

We will now extend the local expressions of the Poisson bivectors defined on the neighborhood of the singularity set to a global Poisson structure $\Pi$ on $X=M^3\times S^1$, whose symplectic foliation is related to the fibration $F\colon X \to {\bf R}\times S^1$. Recall that $F|_{M^3}$ is given by a Bott-Morse function $f\colon M^3 \to {\bf R}$. The rank of $\Pi$ is 2 everywhere on $X$ except at the singularities of $f$, where the rank drops down to zero. The regular fibers of $F$ are 2-dimensional symplectic leaves of $\Pi$.  If $p\in X$ is a critical point of $f$ contained in the singular fiber $F_p$, then $F_p\setminus \{p\}$ is a 2-dimensional symplectic leaf of $\Pi$.  The following construction of $\Pi$ completes the proof of Theorem \ref{thm:Bott-Morse}.
\newline

The idea of the construction is to use the local models of the Poisson structures around the different singularities described in Section \ref{sec:local} as the building blocks for $\Pi$. These bivectors together with a regular Poisson structure coming from the area forms on the $2$--dimensional leaves of the regular part of $\mathcal{F}$ will endow $X$ with a global Poisson structure. We will need to do a smooth interpolation between the singular and the regular Poisson structures. For this we will use the following lemmata (2.8 and 2.9 in \cite{GSV}).

\begin{lemma}
\label{L:1}
Let $(M,\pi)$ be a regular rank 2 Poisson manifold  and $g\in C^\infty(M)$ be 
any non-vanishing function. Then $(M,g\pi)$ is a regular rank 2 Poisson manifold and the leaves of its symplectic foliation coincide with the leaves of the symplectic foliation of $(M,\pi)$.
\end{lemma}

\begin{lemma}
\label{L:2}
Let $\pi_1$ and $\pi_2$ be bivectors that define regular rank 2 Poisson structures on the manifold $M$.  Assume that the symplectic
foliations of $\pi_1$ and $\pi_2$ coincide. Then there
exists a nonvanishing function $g\in C^\infty(M)$ such that $\pi_1=g\pi_2$.
\end{lemma}

\begin{remark}\label{P:pi_F}
Let  $F\colon X\to {\bf R}\times S^1$ be a Bott-Morse foliation with singular set $\mathcal{S}$ as described above. There exists an open set  $W\subset X$   that does not contain any critical points of $F$, and a regular rank 2 Poisson structure $\pi_F$ defined on $W$ such that  the symplectic leaves of $\pi_F$ coincide with the intersection of the fibers of $f$ with $W$. Moreover, the region $W$ satisfies 
$$ X=W\cup U_{\mathcal{S}}.$$
Here $U_{\mathcal{S}}$ is the tubular neighborhood of the singularity set. 
\end{remark}

Recall that the Poisson structures of table (\ref{tablefour}), were defined in the neighborhoods of the singularities. The elements of the singularity set can be assumed to be disjoint. To ease the gluing construction we denote by $\pi_{\mathcal{S}}$ the Poisson bivector defined on the neighborhoods of the components of the singularity set $\mathcal{S}$. That is, $\pi_{\mathcal{S}} \in \Gamma(\Lambda^2 TX)$ is locally defined by one of the seven local expressions described above and is zero everywhere else. The definition of the open subset $W\subset X$ in Remark \ref{P:pi_F} is in terms of the open sets $V_{\mathcal{S}}$ satisfying
$$C\subset V_{\mathcal{S}} \subset U_{\mathcal{S}}\quad\hbox{with}\quad C \in \mathcal{S}.$$
We define
$$\Pi(p)=\begin{cases} \pi_F(p) \qquad \mbox{if} \qquad p\in  W\setminus U_{\mathcal{S}},\\
 \pi_{\mathcal{S}}(p)\qquad \mbox{if} \qquad p\in \overline{V_{\mathcal{S}}} . \end{cases} $$%
This defines $\Pi$ on the complement of the set $W \cap U_{\mathcal{S}}$.  The set $U_{\mathcal{S}}$ is composed of the collection of open sets defined around each of the seven types of singularities. That is, $U_{\mathcal{S}} = U_{\mathcal{S}_1} \cup \dots \cup U_{\mathcal{S}_7}$. Hence, we have that $W \cap U_{\mathcal{S}} = (W \cap U_{\mathcal{S}_1}) \cup (W \cap U_{\mathcal{S}_2}) \cup \dots \cup (W \cap U_{\mathcal{S}_7})  $.  We shall now define $\Pi$ on each of the open sets forming the above union. Since the neighborhoods $U_{\mathcal{S}_i}, i\in \lbrace 1, \dots, 7 \rbrace$ are disjoint, the gluing process from the local Poisson structure around the singularities to a regular Poisson structure is the same for all neighborhoods $U_{\mathcal{S}_i}$. Thus, it is enough to show the construction for one singularity, and we simplify this by considering just $U_{\mathcal{S}}$. The reader might find it useful to 
refer to figure (\ref{F:Venn}) during the following construction, taking into account that the sets $\tau_{j}$ are open and $\tau_{j}\not\subset\tau_{j+1}$ for $j=1,...,4$. 

\begin{figure}
 \centering
  \includegraphics[width=0.5\textwidth]{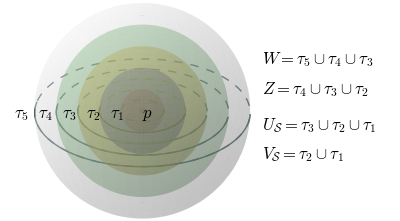}
  \caption{Schematic representation of the gluing process around a singular point $p\in \mathcal{S}$.}
  \label{F:Venn}
\end{figure}

In order to  define $\Pi$ on each connected component of $W\cap U_{\mathcal{S}}$ start by noticing that both bivectors $\pi_F$ and $\pi_{\mathcal{S}}$ were already defined on this set as shown in Section \ref{Poisson} and remark \ref{P:pi_F}.  Moreover, when restricted to this open set, both bivectors define regular rank two Poisson structures possessing the same symplectic foliation. Hence, by Lemma \ref{L:2}, there exists a non-vanishing $C^\infty$ function $g$ such that
\begin{equation*}
\pi_{\mathcal{S}}=g\pi_{F} 
\end{equation*}
on the component of  $ W\cap U_{\mathcal{S}}$ that we are working with.
By changing $\pi_{\mathcal{S}}$ to $-\pi_{\mathcal{S}}$ if necessary, we can assume that $g$ is positive.

Consider a partition of unity, defined by two nonnegative auxiliary smooth cut-off functions $\sigma$ and $\rho$ on a small neighborhood  $Z$ around our connected component of
$\overline{W\cap U_{\mathcal{S}}}$ that satisfy:
\begin{equation*}
\sigma(p)=\begin{cases} 0\; \mbox{if} \; p\notin U_{\mathcal{S}} \\
1\; \mbox{if} \; p\notin W\end{cases} \qquad \rho(p)=\begin{cases} 1\; \mbox{if} \; p\notin U_{\mathcal{S}} \\
0\; \mbox{if} \; p\notin W\end{cases}
\end{equation*}

We can now extend $\Pi$ to $W\cap U_{\mathcal{S}}$ as
\begin{equation*}
\Pi(p)=(g(p)\sigma(p)+\rho(p))\pi_F(p), \qquad \mbox{for} \quad p\in W\cap {\mathcal{S}}.
\end{equation*}
This is a smooth interpolation between the definitions of $\Pi$ on $\overline{V_{\mathcal{S}}}$ 
and on $W\setminus (U_{\mathcal{S}} )$. Indeed, as a point $p\in W\cap U_{\mathcal{S}}$ approaches
$V_{\mathcal{S}}$, the bivector $\Pi(p)$ approaches $\pi_{\mathcal{S}}$. Similarly, as  a point $p\in W\cap U_{\mathcal{S}}$ leaves $U_{\mathcal{S}}$ 
the  bivector $\Pi(p)$ approximates to $\pi_{F}$. Notice that $\Pi$ as defined above is a Poisson structure (satisfies the Jacobi identity) on $W\cap U_{\mathcal{S}}$
in virtue of Lemma \ref{L:1}.

As the function $g\sigma+\rho$ is non-negative, we conclude that the
symplectic leaves of $\Pi$ on $W\cap U_{\mathcal{S}}$ coincide with the symplectic leaves of $\pi_F$.
By Proposition \ref{P:pi_F} these are the pieces of the fibers of $F$ that lie within $W\cap U_{\mathcal{S}}$.

Therefore we have produced a Poisson structure with the claimed properties. If the closure of every symplectic leaf in our construction is compact, then the Poisson structure we obtain is complete. This concludes the proof of Theorem (\ref{thm:Bott-Morse}).

\section{Obstructions to Poisson structures on Bott-Morse foliations}\label{S7}

If we keep the codimension-one hypothesis, then the next dimension where Poisson structures supported on Bott-Morse foliations could be found is dimension 5, with 4-dimensional leaves.\\ 

{\bf Example 3:} Consider $S^4$ canonically embedded as the unit sphere in ${\bf R}^5$. Let $h$ be a height function, defined by projecting $S^4$ onto a closed interval on an axis. Then $-\nabla h$ is a Morse  function from $S^4$ to a closed interval $I$. It has two types of level sets. One type is diffeomorphic to $S^3$, coming from preimages of the interior points of the interval. The others are two points, corresponding to $\partial I$. We now define a Bott-Morse foliation on the smooth 5-manifold $S^4\times S^1$ as follows. Set the leaves of the foliation to be given by $(-\nabla h)^{-1}(t)\times S^1$, for $t$ in $I$. There are two kinds of leaves, corresponding to the two kinds of level sets of $-\nabla h$. The first kind is diffeomorphic to $S^3\times S^1$, the second kind is diffeomorphic to a circle. As the singular components of this foliation are locally modelled by a Morse function, it is an example of a Bott-Morse foliation. Observe that, as $S^3\times S^1$ has trivial second cohomology it can not be symplectic. Therefore this Bott-Morse foliation on $S^4\times S^1$ does not admit a compatible Poisson structure. This can be seen as a special case of Example 2.6 in \cite{SS11}.\\

This example leads to the next:\\

{\bf Question:} Does a codimension-one Bott-Morse foliation $\mathcal{F}$ with symplectic leaves admit a Poisson structure supported on $\mathcal{F}$\,?\\

In higher dimensions there is the further complexity of inequivalent symplectic structures.\\

{\bf Example 4:} As in the previous example, a codimension-one Bott-Morse foliation may be defined on $S^3\times S^2$ with the help of a Morse function on $S^3$. Consider a height function with respect to the standard embedding so that the inverse images are now 2-spheres. The foliation thus defined on $S^3\times S^2$ has leaves diffeomorphic to $S^2\times S^2$, and two singular components whose points as a set are homeomorphic to $S^2$. Let $\omega$ be an area form for $S^2$. Notice that $S^2\times S^2$ may be given symplectic forms $\omega + \lambda \omega$, which are known to produce symplectic structures that are not equivalent when the value of $\lambda$ changes enough \cite{G85}. Here we face a different situation, as the foliation itself is not enough to determine the Poisson structure. There are potentially countably many different Poisson structures that could be associated to the underlying Bott-Morse foliation, coming from the diversity of equivalence classes of symplectic structures on the leaves.\\ 

Given the classification results obtained in \cite{SS09,SS11} there are cohomological restrictions to the existence of Poisson structures on Bott-Morse foliations with singularities of center type. In order for the leaves of the symplectic foliation to be even dimensional, the total space $M$ for the Bott-Morse foliation would have to have dimension $2n+1$,  for $n \in \bf{N}$. A second restriction comes from the topology of the leaves $L$, which are $S^n$-bundles over $N_j$.  For $L$ to admit a symplectic structure, the codimension of $N_j$ can only be 2 or 3, so that the associated sphere bundles are either $S^1$- or $S^2$-bundles. These cases present the possibility to build a compatible symplectic form.  In any other case the  $S^k$--bundles do not admit symplectic structures. 

On one hand if ${\rm codim}(N_j)=2$, it is not clear that the total space $L$ of the bundle can admit a symplectic structure when $\dim(M)\geq 7$. In the case $\dim(M)=5$ there are examples of $S^1$--bundles that admit symplectic structures (see, for example, McMullen-Taubes \cite{MM99} or Friedl-Vidussi \cite{FV11}).

On the other hand if ${\rm codim}(N_j)=3$, that is when $L$ is an $S^2$--bundle, it might be possible to extend our methods but only a rank $2$ Poisson structure could be constructed following our arguments in this paper. A general construction of Thurston provides conditions for total spaces of certain surface bundles to admit symplectic structures via symplectic fibrations \cite{T76}. One of the requirements is that the base of the fibration is symplectic, so an additional obstruction is that $N_j$ admits a symplectic form.  Hence, for $\dim(M) = 5, \mathrm{codim}(N_j) = 3$, it is possible to apply our construction to find new examples of Bott-Morse foliations with center singularities and compatible Poisson structures. 

Moreover, another possible extension of this last idea could involve Lefschetz fibrations with genus $0$ fibers, which are well known to admit symplectic structures on their total spaces.  The work of Donaldson on Lefschetz pencils \cite{D99} asserts that for a suitable cohomology class in the second de Rham cohomology of $L$, there is a symplectic structure on $L$ with symplectic fibers. As Lefschetz pencils are defined over $S^2$, the possibility of constructing a singular Poisson structure that is symplectic on the complement the Bott-Morse singularities would only hold in $\dim(M) = 5$.

\section{A remark on Poisson Cohomology}\label{S8}

Poisson cohomology displays interesting global characteristics of the geometry of Poisson structures. It reveals information about deformations of Poisson structures, which becomes relevant in deformation theory.  In general, calculating Poisson cohomology is very hard as the cohomology groups can be infinite dimensional, and there is no general method to compute them. However, it is known in certain cases. For instance, for symplectic manifolds, the Poisson cohomology is isomorphic to its de Rham cohomology $H_{\pi}^{*}(M) \simeq H_{\textnormal{dR}}^{*}(M)$.  For a compact Lie algebra $\mathfrak{g}$ with corresponding  Lie-Poisson structure $W$ on $\mathfrak{g}^{*}$, denote by $H_{\text{Lie}}^{k}(\mathfrak{g}^{*})$ the Lie algebra cohomology of $\mathfrak{g}$ and by $\text{Cas}(\mathfrak{g}^{*}, W)$ the space of Casimirs of $(\mathfrak{g}^{*}, W)$. In this case $H^{k}_{\pi}(\mathfrak{g}^{*}, W) = H_{\text{Lie}}^{k}(\mathfrak{g}^{*}) \otimes \text{Cas}(\mathfrak{g}^{*}, W) $. 

Next we will describe the Poisson cohomology for the structures described in this work, but first we recall its definition.  Consider the space of multivector fields $\mathfrak{X}^{*}(M) = \Gamma(\Lambda^{*} TM)$ and 
$$\dots \longrightarrow \mathfrak{X}^{k-1}(M) \stackrel{d_\pi}{\longrightarrow} \,  \mathfrak{X}^{k}(M) \stackrel{d_\pi}{\longrightarrow} \, \mathfrak{X}^{k+1}(M) \stackrel{d_\pi}{\longrightarrow} \dots $$ 
The operator $d_{\pi} \colon \mathfrak{X}^{\bullet}(M)\longrightarrow \mathfrak{X}^{\bullet + 1}(M), d_{\pi}(X) = \left[\pi,X \right]_{\mathrm{SN}}$ is a differential of the exterior algebra $\mathfrak{X}(M) = \oplus_{k} \mathfrak{X}^{k}(M)$ and, due to the Poisson condition, it satisfies $d_{\pi}^{2}=0$.   The pair $(\mathfrak{X}(M), d_\pi)$  is called the {\em Poisson} or {\em Lichnerowicz-Poisson cochain complex}, and 
$$H^{k}_{\pi}(M) := \frac{\ker\left( d_{\pi}\colon \mathfrak{X}^{k}(M) \rightarrow \mathfrak{X}^{k+1}(M) \right) }{ \mathrm{Im}\left( d_{\pi}\colon \mathfrak{X}^{k-1}(M) \rightarrow \mathfrak{X}^{k}(M) \right)}$$
with $k\in \mathbb{N}_0$ are called the {\em Poisson cohomology spaces} of $(M, \pi)$. 

Let us comment briefly on the interpretation of the Poisson cohomology groups. The zeroeth Poisson cohomology group $H^{0}_{\pi}(M)$ is generated by Casimir functions, whereas $H^{1}_{\pi}(M)$ measures the Poisson vector fields that are not Hamiltonian.  The cohomology group $H^{2}_{\pi}(M)$ is the quotient of infinitesimal deformations of $\pi$ over trivial deformations, and $H^{3}_{\pi}(M)$ reflects the obstructions to formal deformations of $\pi$. 

The next tables (\ref{tablesix}), (\ref{tableseven}), and (\ref{tableeight}) summarize the Poisson cohomology of the structures associated to Bott-Morse foliations in dimension 3.  The first table presents the cases where the Poisson cohomology is isomorphic to the Lie algebra cohomology. This is a consequence of the defined Lie algebras being compact of semi-simple type (\cite{DZ05}, p. 49).  There are three cases where this does not apply.  When $\dim(N_j)=1$ and the Morse index is 0 or 2, its corresponding Lie algebra is $\mathfrak{e}(2)$ and this is not semisimple. For $\dim(N_j) = 0$ with Morse index 2 and $\dim(N_j)=1$ with Mose index 1 it is not known to us what are the corresponding Lie algebras. Direct computations with linear coefficients give a partial result on $H_{\pi}^{*}(M)$, and we describe the dimensions of the cohomology groups and its generators.  In the following tables (\ref{tablesix}), (\ref{tableseven}), and (\ref{tableeight}), we use the simplified notation $\partial_{ik}:= \frac{\partial}{\partial x_i}\wedge \frac{\partial}{\partial x_k} $ to describe the bivector fields. 

\begin{table}[H]
\centering
\resizebox{.9\textwidth}{!}{%
\begin{tabular}{|c|c|c|c|c|}
\hline
\multicolumn{5}{|l|}{$\dim(N_{j})=0$}                                                                                                                                                                                               \\ \hline
\multirow{3}{*}{Morse Index} & \multirow{3}{*}{Poisson Bivector}                                                & \multirow{3}{*}{Lie Algebra}                   & \multicolumn{2}{c|}{Poisson Cohomology}                          \\
                             &                                                                                  &                                                & \multicolumn{2}{c|}{$H_{\pi}^{\bullet}(M) \cong H_{LA}^{\bullet}(\mathfrak{g})$}  \\ \cline{4-5} 
                             &                                                                                  &                                                & $H_{\pi}^{0}(M)$                 & $H_{\pi}^{k}(M)\quad \textnormal{for } \,  k\geq 1$ \\ \hline

\multirow{3}{*}{0}           & \multirow{6}{*}{$x_{3}\partial_{12} -x_{2}\partial_{13} + x_{1}\partial_{23} $}  & \multirow{6}{*}{$\mathfrak{so}(3)$}            & $\simeq{\bf R},$              & \multirow{3}{*}{0}            \\
                             &                                                                                  &                                                & generated by                     &                               \\
                             &                                                                                  &                                                & $\langle x_{1}^{2}+x_{2}^{2}+x_{3}^{2}\rangle$  &                               \\ \cline{1-1} \cline{4-5} 

\multirow{3}{*}{3}           &                                                                                  &                                                & $\simeq{\bf R},$              & \multirow{3}{*}{0}            \\
                             &                                                                                  &                                                & generated by                     &                               \\
                             &                                                                                  &                                                & $\langle -x_{1}^{2}-x_{2}^{2}-x_{3}^{2}\rangle$ &                               \\ \hline
\multirow{3}{*}{1}           & \multirow{3}{*}{$-x_{3}\partial_{12} +x_{2}\partial_{13} + x_{1}\partial_{23} $} & \multirow{3}{*}{$\mathfrak{sl}(2,{\bf R})$} & $\simeq{\bf R}$               & \multirow{3}{*}{0}            \\
                             &                                                                                  &                                                & generated by                     &                               \\
                             &                                                                                  &                                                & $\langle -x_{1}^{2}+x_{2}^{2}+x_{3}^{2}\rangle$ &                               \\ \hline
\end{tabular}%
}
\vspace{0.5mm}
\caption{Lie algebra and Poisson Cohomology of corresponding Poisson Structures.}
\label{tablesix}
\end{table}

\begin{table}[]
\centering
\resizebox{\textwidth}{!}{%
\begin{tabular}{|c|c|c|c|c|c|}
\hline
\multicolumn{6}{|l|}{$\dim(N_{j})=0$}                                                                                                                                                                                                                                                                                      \\ \hline
\multirow{3}{*}{Morse Index} & \multirow{3}{*}{Poisson Bivector}                                                & \multicolumn{4}{c|}{\multirow{2}{*}{Poisson cohomology with linear coefficients}}                                                                                                                        \\
                             &                                                                                  & \multicolumn{4}{c|}{}                                                                                                                                                                                    \\ \cline{3-6} 
                             &                                                                                  & $H_{\pi}^{0}(M)$                                & $H_{\pi}^{1}(M)$                                                          & $H_{\pi}^{2}(M)$                    & $H_{\pi}^{3}(M)$                     \\ \hline
\multirow{3}{*}{2}           & \multirow{3}{*}{$-x_{3}\partial_{12} -x_{2}\partial_{13} + x_{1}\partial_{23} $} & $\simeq{\bf R}$                              &                                                                           &                                     & \multicolumn{1}{l|}{}                \\
                             &                                                                                  & generated by                                    & 0                                                                         & 0                                   & 0                                    \\
                             &                                                                                  & $\langle -x_{1}^{2}-x_{2}^{2}+x_{3}^{2}\rangle$ & \multicolumn{1}{l|}{}                                                     & \multicolumn{1}{l|}{}               & \multicolumn{1}{l|}{}                \\ \hline
\multicolumn{6}{|l|}{$\dim(N_{j})=1$}                                                                                                                                                                                                                                                                                      \\ \hline
\multirow{3}{*}{0}           & \multirow{6}{*}{$ -x_{2}\partial_{13} + x_{1}\partial_{23} $}                    & \multirow{7}{*}{$\simeq{\bf R},$}                            & \multirow{7}{*}{$\simeq{\bf R}$}                                                        & \multirow{7}{*}{$\simeq{\bf R}$}                  & \multirow{7}{*}{$\simeq{\bf R}$}                   \\
                             &                                                                                  & \multirow{8}{*}{generated by}                                    & \multirow{8}{*}{generated by}                                                              & \multirow{8}{*}{generated by}                        & \multirow{8}{*}{generated by}                         \\
                             &                                                                                  & \multirow{9}{*}{$\langle -x_{1}^{2}+x_{2}^{2}\rangle$}           & \multirow{9}{*}{$\langle x_{1}\partial_{1}+x_{2}\partial_{2}\rangle$} & \multirow{9}{*}{$\langle x_{3}\partial_{12}\rangle$} & \multirow{9}{*}{$\langle x_{3}\partial_{123}\rangle$} \\ \cline{1-1} 
\multirow{3}{*}{2}           &                   														&                              &                                                         &                   &                   \\
                             &                                                                                  &                                     &                                                               &                         &                          \\
                             &                                                                                  &            & \multicolumn{1}{l|}{} &  &  \\ \cline{1-1} \cline{1-2} 
\multirow{3}{*}{1}           & \multirow{3}{*}{$ x_{2}\partial_{13} + x_{1}\partial_{23} $}                    &                              &                                                         &                   &                   \\
                             &                                                                                  &                                     &                                                              &                    &                         \\
                             &                                                                                  &            & \multicolumn{1}{l|}{} &  &  \\ \hline
\end{tabular}%
}
\vspace{1mm}
\caption{Poisson Cohomology with linear coefficients associated to the Poisson structures.}
\label{tableseven}
\end{table}

\begin{table}[]
\centering
\resizebox{\textwidth}{!}{%
\begin{tabular}{|c|c|c|c|c|c|}
\hline
\multicolumn{6}{|l|}{$\dim(N_{j})=0$}                                                                                                                                                                                                                                                                                      \\ \hline
\multirow{3}{*}{Morse Index} & \multirow{3}{*}{Poisson Bivector}                                                & \multicolumn{4}{c|}{\multirow{2}{*}{Poisson cohomology with quadratic coefficients}}                                                                                                                        \\
                             &                                                                                  & \multicolumn{4}{c|}{}                                                                                                                                                                                    \\ \cline{3-6} 
                             &                                                                                  & $H_{\pi}^{0}(M)$                                & $H_{\pi}^{1}(M)$                                                          & $H_{\pi}^{2}(M)$                    & $H_{\pi}^{3}(M)$                     \\ \hline
\multirow{3}{*}{2}           & \multirow{3}{*}{$-x_{3}\partial_{12} -x_{2}\partial_{13} + x_{1}\partial_{23} $} & $\simeq{\bf R}$                              &                                                                           &                                     & \multicolumn{1}{l|}{}                \\
                             &                                                                                  & generated by                                    & 0                                                                         & 0                                   & 0                                    \\
                             &                                                                                  & $\langle -x_{1}^{2}-x_{2}^{2}+x_{3}^{2}\rangle$ & \multicolumn{1}{l|}{}                                                     & \multicolumn{1}{l|}{}               & \multicolumn{1}{l|}{}                \\ \hline
\multicolumn{6}{|l|}{$\dim(N_{j})=1$}                                                                                                                                                                                                                                                                                      \\ \hline
\multirow{3}{*}{0}           & \multirow{6}{*}{$ -x_{2}\partial_{13} + x_{1}\partial_{23} $}                    & \multirow{7}{*}{$\simeq{\bf R},$}                            & \multirow{7}{*}{$\simeq{\bf R}^{2}$}                                                        & \multirow{7}{*}{$\simeq{\bf R}$}                  & \multirow{7}{*}{$\simeq{\bf R}$}                   \\
                             &                                                                                  & \multirow{8}{*}{generated by}                                    & \multirow{8}{*}{generated by}                                                              & \multirow{8}{*}{generated by}                        & \multirow{8}{*}{generated by}                         \\
                             &                                                                                  & \multirow{9}{*}{$\langle -x_{1}^{2}+x_{2}^{2}\rangle$}           & \multirow{9}{*}{$\langle ax^2_{1}\partial_{3}, bx^2_{2}\partial_{3}\rangle$, $a\not=b$} & \multirow{9}{*}{$\langle x_{3}^{2}\partial_{12}\rangle$} & \multirow{9}{*}{$\langle x_{3}^{2}\partial_{123}\rangle$} \\ \cline{1-1} %
\multirow{3}{*}{2}           &                   														&                              &                                                         &                   &                   \\
                             &                                                                                  &                                     &                                                               &                         &                          \\
                             &                                                                                  &            & \multicolumn{1}{l|}{} &  &  \\ \cline{1-1} \cline{1-2} 
\multirow{3}{*}{1}           & \multirow{3}{*}{$ x_{2}\partial_{13} + x_{1}\partial_{23} $}                    &                              &                                                         &                   &                   \\
                             &                                                                                  &                                     &                                                              &                    &                         \\
                             &                                                                                  &            & \multicolumn{1}{l|}{} &  &  \\ \hline
\end{tabular}%
}
\vspace{1mm}
\caption{Poisson Cohomology with quadratic coefficients associated to the Poisson structures.}
\label{tableeight}
\end{table}


\newpage


\begin{thebibliography}{99}

\bibitem{AV17} M. Avenda\~no-Camacho, Y. Vorobiev, {\it Deformations of Poisson structures on fibered manifolds and adiabatic slow-fast systems,} Int. J. Geom. Methods Mod. Phys. 14 (2017), no. 6, 1750086.


\bibitem{B16} J. Bowden, {\em Contact perturbations of Reebless foliations are universally tight,} J. Differential Geom. 104 (2016), no. 2, 219--237.

\bibitem{C84} J.F. Conn, {\em Normal forms for analytic Poisson structures}, Ann. of Math. 119 (1984), 577--601.

\bibitem{BW05} H. Bursztyn, A. Weinstein, {\em Poisson geometry and Morita equivalence. Poisson geometry, deformation quantisation and group representations}, 1--78, London Math. Soc. Lecture Note Ser., 323, Cambridge Univ. Press, Cambridge, 2005.

\bibitem{C07} D. Calegari, {\em Foliations and the geometry of 3-manifolds,} Oxford Mathematical Monographs, Oxford University Press, Oxford, 2007. 

\bibitem{CF11} M. Crainic, R.L. Fernandes, {\em A geometric approach to Conn's linearization theorem,} Ann. of Math. (2) 173 (2011), no. 2, 1121--1139. 

\bibitem{CDMV16} F.A. Cuesta, F. Dal'Bo, M. Martínez, A. Verjovsky, {\em On the construction of minimal foliations by hyperbolic surfaces on 3-manifolds,} 	arXiv:1611.09833 [math.GT].

\bibitem{D99} S. K. Donaldson, {\em Lefschetz Pencils on Symplectic Manifolds}, J. Differ. Geom. 53 (1999) 205--236.

\bibitem{DZ05} J.-P. Dufour, N.T. Zung, \textit{Poisson structures and their normal forms}, Progress in Mathematics, 242. Birkh\"auser Verlag, Basel, 2005.

\bibitem{E16} H. Eynard-Bontemps, {\em On the connectedness of the space of codimension one foliations on a closed 3-manifold}, Invent. Math. 204 (2016), no. 2, 605--670.


\bibitem{FV11} S. Friedl, S. Vidussi,{\em Twisted Alexander polynomials detect fibered 3--manifolds,} Ann. of Math. (2) 173 (2011), no. 3, 1587--1643. 

\bibitem{GSV} L. C. Garc\'ia-Naranjo, P. Su\'arez-Serrato, R. Vera, {\em Poisson structures on smooth 4-manifolds,} Lett. Math. Phys. 105 (2015) 1533--1550.  

\bibitem{GW92} V.L. Ginzburg, A. Weinstein, {\em Lie-Poisson structure on some Poisson Lie groups,} J. Amer. Math. Soc. 5 (1992), no. 2, 445--453. 

\bibitem{G85} M. Gromov, {\em Pseudo holomorphic curves in symplectic manifolds,} Invent. Math. 82 (1985), no. 2, 307--347.

\bibitem{IM03} A. Ibort, D. Mart\'inez Torres, {\em A new construction of Poisson manifolds}, Jour.  Symp. Geom.,  vol.2, no.1, (2003) 83--107.

\bibitem{LGPV13} C. Laurent-Gengoux, A. Pichereau and P. Vanhaecke, {\em Poisson structures}, Grundlehren der Mathematischen Wissenschaften, 347, Springer, Heidelberg, 2013.

\bibitem{L65} W. B. R. Lickorish,  {\em A foliation for 3--manifolds}, Ann. of Math. (2), 82 (1965), 414--420.

\bibitem{MM99} C.T McMullen, C. Taubes, {\em 4-manifolds with inequivalent symplectic forms and 3-manifolds with inequivalent fibrations,} Math. Res. Lett. 6 (1999), no. 5-6, 681--696. 

\bibitem{N65} S. P. Novikov, {\em The topology of foliations,} Trans. Moscow Math. Soc. 14 (1965), 268--304 (Russian), A.M.S. Translation 1967, 268--304 (1965).

\bibitem{SS09} B. Sc\'ardua, J. Seade, {\em Codimension one foliations with Bott-Morse singularities. I}, J. Differential Geom. 83 (2009), no. 1, 189--212. 

\bibitem{SS11} B. Sc\'ardua, J. Seade, {\em Codimension $1$ foliations with Bott-Morse singularities II,} J. Topol. 4 (2011), no. 2, 343--382.

\bibitem{T76} W. P. Thurston, {\em Some simple examples of symplectic manifolds}, Proc. Amer. Math. Soc. 55 (1976), no. 2, 467--468.

\bibitem{V94} I. Vaisman, {\em Lectures on the Geometry of Poisson Manifolds}, Birkh\"auser, Basel, 1994.

\bibitem{V16} T. Vogel, {\em On the uniqueness of the contact structure approximating a foliation,} Geom. Topol. 20 (2016), no. 5, 2439--2573.


\end{thebibliography}
\end{document}